\documentclass[11pt]{amsart}
\usepackage[all]{xy}
\usepackage{amsfonts,amscd,amssymb,amsmath,mathrsfs}
\theoremstyle{plain}
\newtheorem{thm}{Theorem}[section]
\newtheorem{lem}[thm]{Lemma}
\newtheorem{prop}[thm]{Proposition}
\newtheorem{cor}[thm]{Corollary}
\theoremstyle{definition}
\newtheorem{defn}{Definition}[section]

\theoremstyle{remark}
\newtheorem{rmk}{Remark}[section]

\numberwithin{equation}{section}

\begin{document}

\newcommand{\Soc}{\operatorname{Soc}}
\newcommand{\modP}{\mod_{P}\Lambda}
\newcommand{\modl}{\mod \Lambda}
\newcommand{\module}{\operatorname{mod}}
\newcommand{\p}{\operatorname{p}}
\newcommand{\inj}{\operatorname{i}}
\newcommand{\Module}{\operatorname{Mod}}
\newcommand{\Cok}{\operatorname{Coker}}
\newcommand{\Hom}{\operatorname{Hom}}
\newcommand{\h}{\operatorname{h}}
\newcommand{\e}{\operatorname{e}}
\newcommand{\res}{\operatorname{res}}
\newcommand{\Tr}{\operatorname{Tr}}
\newcommand{\TrD}{\operatorname{TrD}}
\newcommand{\rad}{\operatorname {{\bold r}}}
\newcommand{\La}{\operatorname{\Lambda}}
\newcommand\End{\operatorname{End}}
\newcommand\Ext{\operatorname{Ext^1_\Lambda}}
\newcommand\Ex{\operatorname{Ext}}
\newcommand\ann{\operatorname{ann}}
\newcommand\coend{\operatorname{coend}}
\newcommand\Img{\operatorname{Im}}
\newcommand\D{\operatorname{D}}
\newcommand\DTr{\operatorname{DTr}}
\newcommand\Ker{\operatorname{Ker}}
\newcommand\Coker{\operatorname{Coker}}
\newcommand{\tr}{\operatorname {t}}
\newcommand{\ct}{\operatorname{ct}}
\newcommand{\rej}{\operatorname{rej}}
\newcommand{\gd}{\operatorname{{gl.dim}}}
\newcommand{\radic}{\operatorname{rad}}
\newcommand{\add}{\operatorname{add}}
\newcommand{\Ind}{\operatorname{Ind}}
\newcommand{\Sub}{\operatorname{Sub}}
\newcommand{\Fac}{\operatorname{Fac}}
\newcommand{\lap}{\operatorname{l}}
\newcommand{\rap}{\operatorname{r}}
\newcommand{\locn}{\operatorname{lnRep}}
\newcommand{\rep}{\operatorname{Rep}}

\newcommand{\gl}{(\Gamma,\Lambda)}
\newcommand{\poset}{(\Gamma_{0}¥,\Lambda)}
\newcommand{\Stilde}{\tilde S}
\newcommand{\Ttilde}{\tilde T}
\newcommand{\ad}{(+){\rm -admissible}}
\newcommand{\pa}{{\tilde{\mathcal P}}(A)}
\newcommand{\modA}{{\rm mod}A}
\newcommand{\hull}{H_{\Lambda}}
\newcommand{\ds}[1]{\ensuremath{#1}}
\newcommand{\quiver}[2]{\ensuremath{\left( #1, #2 \right) }}
\newcommand{\quiverA}[1]{\ensuremath{\left( A_n, #1 \right) }}
\newcommand{\set}[1]{\ensuremath{\left\{#1\right\}}}

\newcommand{\spec}{(\Gamma,\Lambda,\textrm{\textbf{d}},\mathfrak B)}
\newcommand{\newspec}{(\Gamma,\sigma_x\Lambda)}

\bibliographystyle{plain}
\title[Preprojective roots]{Preprojective roots of Coxeter groups}
\author{Mark Kleiner}
\address{Department of Mathematics, Syracuse University, Syracuse, New 
York 13244-1150}
\email{mkleiner@syr.edu}

\begin{abstract} Certain results on representations of quivers have analogs in the structure theory of general Coxeter groups. A fixed Coxeter element turns the Coxeter graph into an acyclic quiver, allowing for the definition of a preprojective root.  A positive root  is an analog of an indecomposable representation of the quiver. The Coxeter group is finite if and only if every positive root is preprojective, which is analogous to the well-known result that a quiver is of finite representation type if and only if every indecomposable representation is preprojective.  Combinatorics of orientation-admissible words in the graph monoid of the Coxeter graph relates strongly to reduced words and the weak order of the group.
\end{abstract}
\maketitle

\section*{Introduction}

Coxeter groups, in the crystallographic case, have been used in several recent papers on representations of quivers or more general finite dimensional algebras, see for example ~\cite{is2010, ir2011, airt2012, ort2015}.  On the other hand, Pelley and the author used representations of quivers to prove that the powers of a Coxeter element in an infinite irreducible crystallographic group are reduced ~\cite{kp07}, and then Speyer  proved the result for a general Coxeter group ~\cite{spey08}, using combinatorics of ~\cite{kp07} and stripping out the quiver theory.  The current paper shows that certain results on representations of quivers have analogs in the theory of Coxeter groups.  The road from quivers to Coxeter groups goes through the notion of root in view of the results of Kac ~\cite{kac1980}, for a root of a Coxeter group is an analog of a real root of a quiver, and each positive real root of a quiver is the dimension vector of a unique up to isomorphism indecomposable representation. Therefore we view a positive root of a Coxeter group $\mathcal W$ as an analog of an indecomposable representation, and view a finite set of positive roots as an analog of a representation that need not be indecomposable but has no isomorphic direct summands.  

Among the indecomposable representations of an acyclic quiver that correspond to real roots, the most important are preprojective and preinjective representations introduced by Bernstein, Gelfand, and Ponomarev ~\cite{bgp} as those annihilated by a power of the Coxeter functor.  The analog of the Coxeter functor is a Coxeter element, so Igusa and Schiffler ~\cite{is2010} fix a Coxeter element $c\in\mathcal W$ and define a $c$-preprojective  (resp. $c$-projective) root as a positive root sent to a negative root by a positive power of $c$ (resp. by $c$). The element $c$ determines a unique acyclic orientation of the Coxeter graph $\Gamma$ of $\mathcal W$ and thus turns it into a quiver.  The inverse Coxeter element, $c^{-1},$ yields the opposite quiver, so the $c^{-1}$-preprojective roots are analogs of preinjective representations. 

To study $c$-preprojective or $c$-projective roots, we give a different, but equivalent, definition.  Following the suggestion of ~\cite[Note 2, p. 25]{bgp}, we say that a positive root  is $c$-preprojective if there exists a $c$-admissible sequence of vertices of $\Gamma,$ called $(+)$-admissible in  ~\cite{bgp}, for which the associated product of simple reflections  sends the root to a negative root.  In this context  a simple reflection is the analog of a reflection functor, and the advantage is that the collection  of $c$-admissible sequences has a rich combinatorial structure.  The collection has a natural equivalence relation  and a preorder structure that induce  on the set of equivalence classes a partial order closely related to the weak order; thus obtained partially ordered set is a distributive lattice; there is a canonical form for each equivalence class; etc.  These and other results of Pelley, Tyler, and the author ~\cite{kp07, kt05, kt08} hold for an arbitrary acyclic quiver. Based on these results and using the work of Howlett  ~\cite{h} and Speyer ~\cite{spey08}, we show that properties of $c$-preprojective or $c$-projective roots are similar to well-known properties of  preprojective or projective representations of a quiver.

We prove that the Coxeter group $\mathcal W$ is finite if and only if each positive root is $c$-preprojective (Theorem \ref{cprepr5}), and $\mathcal W$ is an elementary abelian 2-group if and only if each positive root is $c$-projective (Proposition \ref{cproj2}).  These statements are analogs of the following well-known  results. A quiver is of finite representation type if and only if each indecomposable representation is preprojective ~\cite[Section VIII.1]{ars}, and a quiver consists of isolated vertices if and only if each indecomposable representation is projective. We obtain (Theorem \ref{cproj}) an explicit description of the $c$-projective roots  similar to that of the indecomposable projective representations  of a quiver ~\cite[Section III.1]{ars}, and show that the linear operators $-c$ and $-c^{-1}$  establish a bijection between the $c$-projective and $c^{-1}$-projective roots.     

To better  handle the combinatorics of $c$-admissible sequences, we use an equivalent but more convenient language of graph monoids  introduced by Cartier and Foata ~\cite{cf1969}.  For any finite undirected graph, the {\em graph  monoid} $\mathfrak M$ is the quotient of the free monoid on the set of vertices modulo the congruence generated by the binary relation $vwRwv,$ for all pairs $\{v,w\}$ of distinct vertices not connected by an edge.  When $\Gamma$ is the graph,  a subset (not a submonoid) $\mathfrak M(c)$ of $\mathfrak M$ corresponds to the equivalence classes of $c$-admissible sequences and, thus, is a distributive lattice having the properties mentioned above.  We say that the elements of $\mathfrak M(c)$ are the {\em $c$-admissible words} of $\mathfrak M.$ The surjective monoid homomorphism $\rho:\mathfrak M\to \mathcal W$ sending each vertex to the associated simple reflection relates combinatorics of $\mathfrak M$ to that of $\mathcal W.$  An element $w\in\mathcal W$ is {\em $c$-admissible} if it has a $c$-admissible preimage under $\rho.$  If $X\in\mathfrak M,$ we say that the word $\rho(X)$ in $\mathcal W$ is {\em reduced} if the length of the element $\rho(X)$ of $\mathcal W$ equals the length of $X.$

Throughout the paper we assume $\mathcal W$ irreducible.  It is straightforward to extend our results to the case when $\mathcal W$ is a finite direct product of irreducible Coxeter groups.

In Section 1 we recall definitions and results about $c$-admissible words.  The notion of {\em principal} $c$-admissible word is important here.  Section 2 deals with $c$-preprojective or $c$-projective roots.  If $\alpha$ is a $c$-preprojective root, we consider the set of elements $X\in\mathfrak M(c)$ for which $\rho(X)\alpha$ is a negative root, and show that the set is a sublattice of $\mathfrak M(c)$ with a unique least element $W_\alpha,$ which must be a principal word.  Likewise, the set of elements of $\mathfrak M(c)$ sending to a negative root each element of a finite set $\Psi$ of $c$-preprojective roots is a sublattice with a unique least element $W_\Psi,$ which must be a join of principal  words.  The latter two statements are parts of Theorem \ref{cprepr2}, which plays a major role in the paper.  The section also contains various properties and characterizations of $c$-preprojective or $c$-projective roots that are analogs of well-known results on representations of quivers.  Section 3 presents the main result, a characterization of finite Coxeter groups in terms of $c$-preprojective roots.  In Section 4 we relate the partial order on $\mathfrak M(c)$ to the left weak order on the set of $c$-admissible elements of $\mathcal W,$ and show among other things that if $X\in\mathfrak M(c),$ then the word $\rho(X)$ is reduced if and only if $X=W_\Psi,$ where  $\Psi$ is a finite {\em independent} set of $c$-preprojective roots; here $\Psi$ is  independent if the decomposition of $W_\Psi$  as a join of principal words has the smallest possible number of terms. Combining these results with a simple inductive construction that produces the canonical form of each principal word in $\mathfrak M(c)$ ~\cite{kt05}, we hope  to continue our study of reduced $c$-admissible words in $\mathcal W.$ 

\section{Admissible words of a graph monoid}\label{prelim}

We begin by recalling some facts, definitions, and notation, using freely  ~\cite{bgp,humph90,kp07, kt05}.  Denote  by $|S|$ the cardinality of a set $S.$\vskip.05in

Given a finite undirected graph $\Gamma=(\Gamma_0,\Gamma_1)$ with the set of vertices $\Gamma_0,$ the set of edges $\Gamma_1,$ and no loops, denote by $\mathfrak M=\mathfrak M_\Gamma$ the graph monoid of $\Gamma$ ~\cite{cf1969}, which is the quotient of the free monoid $\Gamma_0^*$ on the set $\Gamma_0$ modulo the congruence generated by the binary relation $vwRwv,$ for all pairs $\{v,w\}$ of distinct vertices not connected by an edge. The elements of $\mathfrak M$ are all words $X=x_l\dots x_1,\,l\ge0,$ with $x_j\in\Gamma_0$ for all $j$  (this includes the empty word $1$),  the binary operation is concatenation, and two words are equal as  elements of $\mathfrak M$ if and only if one of the words can be obtained from the other by a finite number of interchanges of adjacent letters that are vertices not connected by an edge.  The following notions are well defined. The {\em length} of $X$ is $l=\ell(X).$  The {\em support} of $X,$ Supp$\,X,$ is  the set of distinct vertices among $x_j,\,1\le j\le l.$ The {\em multiplicity} of $v\in\Gamma_0$ in $X,$ $m_X(v),$ is the (nonnegative) number of times $v$ appears among the $x_j,$    and the element $X$ is {\em multiplicity-free} if $m_X(v)\le1$ for all $v\in\Gamma_0.$ The {\em transpose} of $X$ is $X^T=x_1\dots x_l.$

We relate the elements of $\mathfrak M$  to the sequences of vertices of $\Gamma$  by recalling the equivalence relation $\sim$ on the set of all sequences introduced in ~\cite[Definition 1.2]{kt05}.  For any  sequences $U$ and $V,$  set $UrV$ if and only if $U=x_1,\dots, x_i,x_{i+1},\dots, x_l,$  $V=x_1,\dots, x_{i+1},x_i,\dots, x_l,$ and no edge of $\Gamma$ joins $x_i$ and $x_{i+1}.$ Then $\sim$ is the reflexive and transitive closure of the symmetric binary relation $r.$ Every sequence of vertices $x_1,\dots,x_l$ gives rise to the element  $X=x_l\dots x_1$ of $\mathfrak M,$ and $X=Y=y_m\dots y_1$ if and only if the sequences $x_1,\dots,x_l$ and $y_1,\dots,y_m$ are equivalent under $\sim.$   Clearly, the equivalence $\sim$ corresponds to the aforementioned congruence on $\Gamma^*_0,$ so every statement from ~\cite{kt05, kp07,kt08} about the equivalence classes of $\sim$ translates verbatim into a statement about the elements of $\mathfrak M.$

Let $X,Y\in \mathfrak M.$ We set $Y\preceq X$ if $X=UY$ for some $U\in\mathfrak M$ ~\cite[Definition 2.1]{kt05}, and we write $Y\prec X$ if $Y\preceq X$ and $Y\ne X.$ It is straightforward that the binary relation $\preceq$ is a partial order, with 1 being the least element of  the  {\em partially ordered set (poset)}  $\mathfrak M.$   The poset satisfies the {\em descending chain condition}.

 An {\em orientation} $\La$ of $\Gamma$ consists of two functions, $s:\Gamma_1\to\Gamma_0$ and $e:\Gamma_1\to\Gamma_0,$ assigning to each edge $a\in\Gamma_1$ its starting point $s(a)$ and endpoint $e(a).$ The pair $(\Gamma,\La)$ is a {\em quiver} (directed graph).  In the quiver, each edge $a\in\Gamma_1$ becomes an {\em arrow} $a:s(a)\to e(a)$ {\em from}  $s(a)$ {\em to} $e(a).$  Denote by $\La^{op}$ the orientation obtained by reversing the direction of each arrow of $(\Gamma,\La).$  There results the {\em opposite quiver} of $(\Gamma,\La),$ which we denote by $(\Gamma,\La)^{op}=(\Gamma,\La^{op}).$ 
  
 For each $x\in\Gamma_0$ and each orientation $\La,$ denote by $x\cdot\La$ the orientation obtained from $\La$ by reversing the direction of each arrow incident to $x$ and preserving the remaining arrows.  This extends uniquely to a left action of the graph monoid $\mathfrak M$ on the (finite) set of all orientations: if $X=x_l\dots x_1$ then $X\cdot\La=x_l\cdot(\dots (x_1\cdot\La)\dots).$

  A {\em path}  in the quiver $(\Gamma,\La)$ is either a  {\em nontrivial path}, or a {\em trivial path}. A nontrivial path  is a word  $p=a_t\dots a_1,\,t>0,$ with $a_j\in\Gamma_1$ and $e(a_j)=s(a_{j+1}),$ $1\le j<t;$ here $t=\ell(p)$ is the {\em length} of $p.$   By definition, $s(a_1)$ is the starting point, and $e(a_t)$ is the endpoint, of $p.$   One writes $p:s(p)\to e(p)$ and says that $p$ is a {\em path  from}  $s(p)$ {\em to} $e(p).$ There are precisely $|\Gamma_0|$ trivial paths: for each $x\in\Gamma_0,$ the trivial path $e_x$ at $x$ is defined by $s(e_x)=e(e_x)=x$ and $\ell(e_x)=0.$  The paths compose as follows.  For any path $r$ one sets $e_{e(r)}r=re_{s(r)}=r.$ If $p=a_t\dots a_1$ and $q=b_u\dots b_1$ are nontrivial paths satisfying $s(q)=e(p),$ then $qp=b_u\dots b_1a_t\dots a_1.$ A path $p$ is an {\em oriented cycle} if $\ell(p)>0$ and $s(p)=e(p).$ We consider only {\em acyclic quivers}, i.e., those without oriented cycles. Then $\Gamma_0$ becomes a poset by setting $x\le y$ if there exists a path from $x$ to $y.$  We denote this poset by $(\Gamma_0,\La).$
  
For the remainder of this section we fix an acyclic quiver $(\Gamma,\La).$
 
Recall that a subset $Q$ of a poset $P$ is an {\em ideal} if $x\in Q$ and $y\le x$ imply $y\in Q,$ and $Q$ is a {\em filter} if $x\in Q$ and $y\ge x$ imply $y\in Q.$  The {\em principal ideal} (resp. {\em principal filter}) generated by $x$ is $(x)=\{y\in P\,|\, y\le x\}$ (resp. $\langle x\rangle=\{y\in P\,|\, y\ge x\}$).
 
A vertex $x$ of $(\Gamma, \La)$ is a {\em sink}  (resp. {\em source}) if it is a maximal (resp. minimal) element of the poset  $(\Gamma_0,\La).$ If $x$ is a sink or source, the quiver $(\Gamma, x\cdot\La)$ is acyclic.  An element $X=x_l\dots x_1$ of $\mathfrak M$ is $\La$-{\em admissible}, or (+)-admissible in the terminology of ~\cite{bgp}, if $x_1$ is a sink in $(\Gamma,\La),$ $x_2$ is a sink in $(\Gamma,{x_1}\cdot\La),$ $x_3$ is a sink in $(\Gamma,{x_2}{x_1}\cdot\La),$ and so on. By ~\cite[proof of Lemma 1.2, p. 24]{bgp}, the latter definition is independent of the choice of a word representing the element of $\mathfrak M.$ An element $K\in\mathfrak M$ is {\em complete} if it is multiplicity-free, $\La$-admissible, and Supp$\,K=\Gamma_0.$  Complete elements exist, and we always denote a complete element by $K.$ For all integers $t\ge0$, $K^t$ is $\La$-admissible and $K^t\cdot\La=\La.$   Denote by $\mathfrak M(\La)$ the subset of $\mathfrak M$ consisting of all $\La$-admissible elements.  If $X\in\mathfrak M(\La),$ the quiver $(\Gamma,X\cdot\La)$ is acyclic.

\begin{rmk}\label{cprepr.1} Suppose $X=x_l\dots x_1$ belongs  to $\mathfrak M(\La).$\vskip.05in

(a) If $X=ZY,$ then  $Y\in\mathfrak M(\La)$  and $Z\in\mathfrak M(Y\cdot\La).$\vskip.05in

(b) Let  $x_j$ be a sink, $0<j\le l,$ and let $i$ be the smallest index satisfying $x_j=x_i,\,0<i\le j.$ Then no edge of $\Gamma$ joins $x_i$ and $x_h$ if $h<i,$ and we have $X= x_l\dots x_{i+1}x_{i-1}\dots x_1x_i.$ 

Indeed, if $h$ is the smallest index for which an edge joins $x_i$ and $x_h,$ there is an arrow $a:x_h\to x_i.$  If $h<i,$ the arrow $a$ is not affected by  the successive reversing of the direction of the arrows incident to $x_1,\dots,x_{h-1}.$ Hence $x_h$ is not a sink in the quiver $(\Gamma,x_{h-1}\dots{x_1}\cdot\La),$ a contradiction.\end{rmk}

The following statement quotes \cite[Proposition 1.3]{kt05} and ~\cite[Proposition 3.3]{kp07}.

\begin{prop}\label{filter}  
\begin{itemize}
\item[(a)]  Let $\Theta$ be a subset of $\Gamma_0.$  There exists an element $X\in\mathfrak M(\La)$ satisfying $\Theta=\mathrm{Supp}\,X$ if and only if $\Theta$ is a filter of $(\Gamma_0,\La).$ If $\Theta$ is a filter of $(\Gamma_0,\La),$ there exists a unique multiplicity-free $X\in\mathfrak M(\La)$ satisfying $\Theta=\mathrm{Supp}\,X.$
\item[(b)]  For all $X,Y \in\mathfrak M(\La),$ $X\preceq Y$ if and only if $m_X(v)\le m_Y(v)$ for all $v\in\Gamma_0.$
\end{itemize}
\end{prop}

We quote ~\cite[Definition 3.4, Proposition 3.6, Theorem 3.7, and the proof of part (3) of the latter]{kp07}.

\begin{thm}\label{lattice}
\begin{itemize}
\item[(a)]  The poset $(\mathfrak M(\La),\preceq)$ is a lattice with meet $\wedge$ and join $\vee.$
\end{itemize}
\vskip.02in\noindent For the remaining assertions assume $X,Y \in\mathfrak M(\La).$ 
\begin{itemize}
\item[(b)]  $X= V(X\wedge Y)$ and $Y= W(X\wedge Y),$ where  the elements $V$ and $W$ of $\mathfrak M$ are uniquely determined  and $\mathrm{Supp}\,V\cap\, \mathrm{Supp}\,W=\emptyset.$
\item[(c)]  No edge of $\Gamma$ joins a vertex from $\mathrm{Supp}\,V$ and a vertex from $\mathrm{Supp}\,W.$  Hence $VW= WV.$
\item[(d)] $X\vee Y= VW(X\wedge Y)= WX= VY.$
\end{itemize}
\end{thm}

The following definition quotes  ~\cite[Definitions 1.5 and 2.2]{kt05}.

\begin{defn}\label{filter1}  The {\em hull} of a  filter $\Theta$ of $(\Gamma_0,\La)$   is the smallest filter $H_{\La}(\Theta)$ of  $(\Gamma_0,\La)$ that contains $\Theta,$ as well as each vertex of $\Gamma_0\setminus\Theta$ joined by an edge to a vertex in $\Theta.$

An  element $X\in\mathfrak M$ is {\em $\La$-principal of size} $r>0$ if $X= X_r\dots X_1,$ where:\vskip.02in 

(i) $X_j\in\mathfrak M$ is multiplicity-free, $0<j\le r;$  

(ii) Supp$\,X_r=\langle x\rangle$ is the principal filter of $(\Gamma_0,\La)$ generated by some $x\in \Gamma_0;$ and 

(iii) $H_{\La}(\mathrm{Supp}\, X_{j+1})=\mathrm{Supp}\,X_j,$ $0<j<r.$  
\vskip.02in\noindent If (i) - (iii) hold, we write $X=W_{r,x}$  and say that $X_r\dots X_1$ is the {\em canonical form} of $X.$ 

By definition, the empty word $1$ is $\La$-principal of size 0. Denote by ${\mathfrak P}(\La)$ the set of $\La$-principal elements of $\mathfrak M.$
\end{defn}

If $X\in{\mathfrak P}(\La),$ the full subgraph of $\Gamma$ determined by Supp$\,X$ is connected.\vskip.03in

We now quote  ~\cite[Proposition 1.11 (b) and Corollaries 2.2 and 2.3]{kt05}.

\begin{prop}\label{filter2}
\begin{itemize}
\item[(a)]  ${\mathfrak P}(\La)\subset\mathfrak M(\La).$
\item[(b)]  Let $Y= Y_q\dots Y_1$ be an element of $\mathfrak M$ where $Y_i\in\mathfrak M$ is multiplicity-free, $0<i\le q;$ $\mathrm{Supp}\,Y_i$ is a filter of $(\Gamma_0,\La);$ and $H_{\La}(\mathrm{Supp}\, Y_{i+1})\subset\mathrm{Supp}\,Y_i,\, 0<i< q.$  Then $Y\in\mathfrak M(\La),$ and $W_{r,x}\preceq Y$ if and only if $r\le q$ and $x\in\mathrm{Supp}\,Y_r.$ In particular, $W_{r,x}\preceq K^r$ for any complete element $K.$ 
\item[(c)] $W_{r,x}= W_{q,y}$ if and only if $r=q$ and $x=y.$
\item[(d)] If $W_{r,x}=x_l\dots x_1$ then $x_l=x.$
\end{itemize}
\end{prop}

\begin{defn}\label{independent.1}  A finite subset $\{X_1,\dots, X_m\}$ of $\mathfrak P(\La)$ is {\em independent} if whenever $X_1\vee\dots\vee X_m=Y_1\vee\dots\vee Y_q$ with $Y_j\in\mathfrak P(\La),\,0<j\le q,$ then $m\le q.$
\end{defn}

 We quote ~\cite[Proposition 4.2(3)]{kp07}.
 
 \begin{prop}\label{independent.2}  For all $X\in\mathfrak M(\La)$ there exists an independent subset $\{X_1,\dots, X_m\}$ of $\mathfrak P(\La)$ satisfying $X=X_1\vee\dots\vee X_m.$ If $\{Y_1,\dots, Y_q\}$ is an independent subset of $\mathfrak P(\La)$ satisfying   $X=Y_1\vee\dots\vee Y_q,$ then $q=m$ and there exists a reindexing so that $Y_i= X_i$ for all $i.$
 \end{prop}
 
Note that  ~\cite[Proposition 4.2(1)]{kp07} explains how to construct the words $X_i$ from the given word $X$ in the above proposition.

\section{Preprojective roots}

We now consider quivers arising from Coxeter groups, using freely the terminology of ~\cite{humph90}.
 
Let $\mathcal W=(\mathcal W,S)$ be a {\em Coxeter system},  where $\mathcal W$ is a group with a finite set of generators $S,$ and let $n=|S|.$  The defining relations are $(ss')^{m(s,s')}=1,$ with $s,s'\in S$ and $m(s,s')<\infty,$ where the $m(s,s')$ are the entries of  a {\em Coxeter matrix} $M=(m(s,s'))_{s,s'\in S}.$ Here $M$ is a symmetric $n\times n$ matrix with $m(s,s')\in\mathbb Z\cup\{\infty\};$ $m(s,s)=1$ for all $s\in S;$ and $m(s,s')>1$ whenever $s\ne s'.$ Denote by $\Gamma=(\Gamma_0,\Gamma_1)$  the {\em Coxeter graph} of $\mathcal W.$ The vertices of $\Gamma$ are defined by a bijection $\rho:\Gamma_0\to S.$ There exists one, and only one, edge joining vertices $x$ and $y$ if and only if  $2<m(\rho(x),\rho(y))\le\infty.$  In this paper we assume that $\mathcal W$ is {\em irreducible}, i.e., that the graph $\Gamma$ is {\em connected}.  We denote by the same letter $\rho$ a unique (surjective) monoid homomorphism $\mathfrak M\to \mathcal W$ induced by the bijection $\rho;$ here $\mathfrak M$ is the graph monoid of the Coxeter graph $\Gamma.$ The {\em length} of $w\in\mathcal W$ is $\ell(w)=\ell(X)$ if $w=\rho(X)$ and $\ell(X)\le\ell(Y)$ for all $Y\in\mathfrak M$ satisfying $w=\rho(Y).$  If $\ell(w)=\ell(X)$ and $X=x_l\dots x_1,$ we say that the word $\rho(X)=\rho(x_l)\dots \rho(x_1)$ in $\mathcal W$ is  {\em reduced}, and we also say that $\rho(X)$ is a {\em reduced expression for} $w.$  

\begin{rmk}\label{reduced1} No edge of $\Gamma$ joins vertices $x$ and $y$ if and only if $\rho(x)\rho(y)=\rho(y)\rho(x).$ Hence $X= Y$ implies $\rho(X)=\rho(Y),$ i.e., the map $\rho:\mathfrak M\to \mathcal W$ is well defined.
\end{rmk}

\begin{defn}\label{cprepr1}  Let $s_1,s_2,\dots,s_n$ be the elements of $S$ in some order, and let $\Gamma_0=\{v_1,\dots,v_n\}$ where $s_j=\rho(v_j)$ for all $j.$ The element $c=s_n\dots s_1$  is a {\em Coxeter element} of $\mathcal W.$  The $c$-{\em orientation} of $\Gamma$ is the orientation for which every arrow ${v_j}\to {v_i}$ satisfies $j>i.$  There result an acyclic quiver $(\Gamma,c),$ where we denote the orientation by the same letter $c;$  a poset $(\Gamma_0,c);$ the subset $\mathfrak M(c)$ of $\mathfrak M$ consisting of all {\em $c$-admissible elements}; and the subset $\mathfrak P(c)$ of $\mathfrak M(c)$ consisting of all {\em $c$-principal elements}. An element $w\in\mathcal W$ is  {\em $c$-admissible} if one of its preimages under $\rho$ is  $c$-admissible.
\end{defn}

\begin{rmk}\label{cprepr1.1} If $\Gamma$ is the Coxeter graph and a quiver $(\Gamma,\La)$ is acyclic, then there exists a Coxeter element $c$ for which $\La$ is the $c$-orientation.   If $c$ and $d$ are Coxeter elements, one can verify that $c=d$ if and only if the $c$-orientation and the $d$-orientation of $\Gamma$ coincide, that is, if and only if $(\Gamma,c)=(\Gamma,d)$ as quivers.
\end{rmk}

For the rest of the paper we fix an irreducible Coxeter system $\mathcal W=(\mathcal W,S),$ a Coxeter element $c=s_{n}\dots s_{1},$ and a bijection $\rho:\Gamma_0\to S$ given by $\rho(v_j)=s_j,\,j=1,\dots,n.$

\begin{rmk}\label{cprepr1.11} Let $X=x_l\dots x_1$ be in $\mathfrak M(c).$ 

(a)  The element $K=v_n\dots v_1$ is complete $c$-admissible, $\rho(K)=c,$  and  $XK\in\mathfrak M(c).$  The element $K^T=v_1\dots v_n$ is complete $c^{-1}$-admissible with $\rho(K^T)=c^{-1}.$  We have $(\Gamma,c)^{op}=(\Gamma,c^{-1}).$\vskip.05in

(b) If $x_1=v_j,$ for some $j,$ then $K= v_n\dots v_{j+1}v_{j-1}\dots v_1v_j$ by Remark \ref{cprepr.1}(b), and $c=s_n\dots s_{j+1}s_{j-1}\dots s_1s_j$ by (a).  Therefore $s_jcs_j=s_js_n\dots s_{j+1}s_{j-1}\dots s_1$ is a Coxeter element and ${x_1}\cdot c=s_jcs_j.$  By induction, $\rho(X)c\rho(X^T)$ is a Coxeter element and $X\cdot c=\rho(X)c\rho(X^T).$  

Thus the fact that all Coxeter elements are conjugate if $\Gamma$ is a tree, is a consequence of ~\cite[Theorem 1.2, part 1)]{bgp}, saying that if $\La$ and $\La'$ are orientations of a tree $\Gamma,$ then $\La'=Y\cdot\La,$ for some $Y\in\mathfrak M(\La).$\end{rmk} 

Let   $\mathbb V$ be a real vector space of dimension $n$ with a formal basis  $\{\alpha_s\,|\,s\in S\}$ and symmetric bilinear form $B$ given by $B(\alpha_s,\alpha_{s'})=-2\cos\dfrac\pi{m(s,s')},\,s,s'\in S,$  where $\dfrac\pi\infty=0$ by convention. The values of $B$ are twice those of the bilinear form defined in ~\cite[p. 109]{humph90}.  We need the modification in order to simplify the forthcoming explicit description of $c$-projective roots in Theorem \ref{cproj}(b).  For any $\alpha=\underset{s\in S}\sum c_s\alpha_s$ in $\mathbb V,$  $c_s\in\mathbb R$ is the {\em $s$-coordinate} of $\alpha,$ and the {\em support} of $\alpha$ is Supp$\,\alpha=\{x\in \Gamma_0\,|\,c_{\rho(x)}\ne0\}.$

The group $\mathcal W$ acts on $\mathbb V$ by $s\lambda=\lambda-B(\alpha_s,\lambda)\alpha_s,\, s\in S, \lambda\in\mathbb V,$ and by extending the action from the generators to the whole group.  The action preserves the bilinear form $B.$  A vector $\alpha=w(\alpha_s),$ for some $w\in\mathcal W,s\in S,$ is a {\em root}. The element $s_\alpha=wsw^{-1}$  of $\mathcal W,$ which does not depend on the choice of either $w$ or $s,$ is the {\em reflection associated with} $\alpha.$  The basis vectors $\alpha_s$ are the {\em simple roots}. The elements of $S$ are the {\em simple reflections}.  A root $\alpha$ is {\em positive}, $\alpha>0$ (resp. {\em negative}, $\alpha<0$) if all of its coordinates are nonnegative (resp. nonpositive).  Every root is either positive or negative.  Denote by $\Phi$ the {\em root system} of  $\mathcal W$, which is the set of all roots; denote by $\Phi^+$ the set of all positive roots; and denote by $T$ the set of all reflections in $\mathcal W.$\vskip.05in

We quote ~\cite[p. 372, bottom]{brink1998}.

\begin{prop}\label{connected}  For all $\alpha\in\Phi$ the full subgraph of $\Gamma$ determined by $\mathrm{Supp}\,\alpha$ is connected.
\end{prop}

\begin{defn}\label{cprepr1.2}   An element $X\in\mathfrak M$ {\it negates} a root $\alpha>0$  if $\rho(X)\alpha<0,$ and $X$ negates a finite subset $\Psi$ of $\Phi^+$ if $X$ negates each element of $\Psi.$ Denote by $\mathfrak N(\alpha)$ (resp. $\mathfrak N(\Psi)$) the set of all elements of $\mathfrak M$ that negate $\alpha$ (resp. $\Psi$).  We call $X$  a {\it minimal element} of $\mathfrak N(\alpha)$ (resp. $\mathfrak N(\Psi)$)  if $X$ is minimal with respect to the induced partial order $\preceq$ on $\mathfrak N(\alpha)$  (resp. $\mathfrak N(\Psi)$).
\end{defn}

\begin{rmk}\label{cprepr1.3} Let $X\in\mathfrak N(\alpha)$ for some $\alpha\in\Phi^+.$ 

(a) Supp$\,\alpha\subset\mathrm{Supp}\,X.$

(b) Suppose $X$ is a minimal element of $\mathfrak N(\alpha).$ If $X= ZY$ and $Z\ne1,$ then $\rho(Y)\alpha>0$ and $Z$ is a minimal element of $\mathfrak N(\rho(Y)\alpha).$
\end{rmk}

\begin{lem}\label{disconnected}  For $V,W\in\mathfrak M,$ suppose that $\mathrm{Supp}\,V\cap\,\mathrm{Supp}\,W=\emptyset$ and no edge of $\Gamma$ joins a vertex from $\mathrm{Supp}\,V$ and a vertex from $\mathrm{Supp}\,W.$  If $\alpha\in\Phi^+,$ then $\rho(V)\alpha>0$ and $\rho(W)\alpha>0$ if and only if $\rho(VW)\alpha>0.$ 
\end{lem}

\begin{proof}  Remark \ref{reduced1} says that \hskip.05in$\rho(VW)\alpha=\rho(WV)\alpha=\rho(V)[\rho(W)\alpha]=\rho(W)[\rho(V)\alpha].$\vskip.03in

Assume $\rho(V)\alpha>0$ and $\rho(W)\alpha>0.$ If $\rho(VW)\alpha<0,$  then Remark \ref{cprepr1.3}(a) says that Supp$\,\alpha\subset\mathrm{Supp}\, VW=\mathrm{Supp}\, V\cup\,\mathrm{Supp}\, W,$ whence Supp$\,\alpha=(\mathrm{Supp}\,\alpha\cap\mathrm{Supp}\, V)\cup(\mathrm{Supp}\, \alpha\cap\mathrm{Supp}\, W).$ Since no edge of $\Gamma$ joins a vertex from $\mathrm{Supp}\,\alpha\cap\mathrm{Supp}\, V$ and a vertex from $\mathrm{Supp}\,\alpha\cap\mathrm{Supp}\, W,$ then Proposition \ref{connected} says that either $\mathrm{Supp}\,\alpha\cap\mathrm{Supp}\, V=\emptyset,$ or $\mathrm{Supp}\,\alpha\cap\mathrm{Supp}\, W=\emptyset.$  Say, the latter holds. Then Supp$\,\alpha\subset\mathrm{Supp}\, V$ so that the action of $\rho(W)$ does not change the positive coordinates of $\rho(V)\alpha.$  We obtain $\rho(VW)\alpha>0,$ a contradiction.

Conversely, assume $\rho(VW)\alpha>0.$  If, say, $\rho(V)\alpha<0,$ then Supp$\,\alpha\subset\mathrm{Supp}\,V$ whence Supp$\,\alpha\cap\mathrm{Supp}\,W=\emptyset,$ so that the action of $\rho(W)$ does not change the negative coordinates of $\rho(V)\alpha.$ Hence $\rho(VW)\alpha<0,$ a contradiction.
\end{proof}

\begin{defn}\label{cprepr1.4} A root $\alpha>0$  is {\em $c$-preprojective} (resp. $c$-{\it projective}) if  some $X\in\mathfrak M(c)$ (resp. multiplicity-free $X\in\mathfrak M(c)$) negates $\alpha.$   Denote by $\mathcal P(c)$ the subset of $\Phi^+$ consisting of all $c$-preprojective roots.
\end{defn}

\begin{rmk}\label{cprepr1.9} A root $\alpha>0$  is $c$-preprojective if and only if $w(\alpha)<0$ for some $c$-admissible $w\in\mathcal W.$
\end{rmk}

\begin{thm}\label{lattice1}  Let $\alpha\in\mathcal P(c)$ and let $\Psi$ be a finite nonempty subset of $\mathcal P(c).$ Then $\mathfrak N(\alpha)\cap\mathfrak M(c)$ and $\mathfrak N(\Psi)\cap\mathfrak M(c)$ are sublattices of $\mathfrak M(c).$
\end{thm}

\begin{proof}  Let $X,Y\in \mathfrak N(\alpha)\cap\mathfrak M(c).$   By Theorem \ref{lattice}(b),  $X= V(X\wedge Y)$ and $Y= W(X\wedge Y)$  where $\mathrm{Supp}\,V\cap\, \mathrm{Supp}\,W=\emptyset.$ 

If $\rho(X\wedge Y)\alpha>0,$ then  $V,W\in\mathfrak N(\rho(X\wedge Y)\alpha).$ By Remark \ref{cprepr1.3}(a),  Supp$\,\rho(X\wedge Y)\alpha\subset\mathrm{Supp}\,V$ and Supp$\,\rho(X\wedge Y)\alpha\subset\mathrm{Supp}\,W,$ whence $\mathrm{Supp}\,V\cap\mathrm{Supp}\,W\ne\emptyset,$ a contradiction. Thus $\rho(X\wedge Y)\alpha<0,$ that is $X\wedge Y\in\mathfrak N(\alpha)\cap\mathfrak M(c).$

We have just proved that $-\rho(X\wedge Y)\alpha>0.$ By assumption, $\rho(V)[-\rho(X\wedge Y)\alpha]>0$ and $\rho(W)[-\rho(X\wedge Y)\alpha]>0.$  According to Theorem \ref{lattice}(c), no edge of $\Gamma$ joins a vertex from $\mathrm{Supp}\,V$ and a vertex from $\mathrm{Supp}\,W.$  Using  Theorem \ref{lattice}(d) and Lemma \ref{disconnected}, we get  $$\rho(X\vee Y)\alpha=\rho(VW)[\rho(X\wedge Y)\alpha]=-\rho(VW)[-\rho(X\wedge Y)\alpha]<0.$$  Thus $X\vee Y\in\mathfrak N(\alpha)\cap\mathfrak M(c).$ We have proved that $\mathfrak N(\alpha)\cap\mathfrak M(c)$ is a sublattice of $\mathfrak M(c).$

If $\Psi=\{\alpha_1\dots,\alpha_m\}$ then $\mathfrak N(\Psi)=\overset{m}{\underset{i=1}\cap} \mathfrak N(\alpha_i).$ Hence  $\mathfrak N(\Psi)\cap\,\mathfrak M(c)=\overset{m}{\underset{i=1}\cap}\left[ \mathfrak N(\alpha_i)\cap\,\mathfrak M(c)\right]$ is a sublattice because the set of all sublattices of a lattice is closed under intersections.
\end{proof}

Denote by $\mathfrak f(c)$ the set of finite subsets of $\mathcal P(c).$
\begin{thm}\label{cprepr2}  In the setting of Theorem \ref{lattice1}:
\begin{itemize}
\item[(a)] The lattice $\mathfrak N(\alpha)\cap\mathfrak M(c)$ (resp.  $\mathfrak N(\Psi)\cap\mathfrak M(c)$) has  a unique least element $W_\alpha$ (resp.  $W_\Psi$),  which is a minimal element of $\mathfrak N(\alpha)$ (resp. $\mathfrak N(\Psi)$).  We obtain a map $\nu:\mathcal P(c)\to\mathfrak M(c)$ given by $\nu(\alpha)=W_\alpha,$ which extends to the map $\xi:\mathfrak f(c)\to\mathfrak M(c)$ given by $\xi(\Psi)=W_\Psi.$
\item[(b)]   $\Img\nu\subset\mathfrak P(c).$ Hence  $W_{\alpha}=W_{r,x}$ for uniquely determined $r>0, \,x\in \Gamma_0.$  
\item[(c)]   $\alpha=\rho\left(W_{r,x}^T\right)\left(-\alpha_{\rho(x)}\right),$ and $\rho\left(W_{i,x}^T\right)\left(-\alpha_{\rho(x)}\right)>0$ if $0<i< r.$ 
\item[(d)] The map $\nu:\mathcal P(c)\to\mathfrak M(c)$ is injective.
\end{itemize}
\end{thm}

\begin{proof} (a) Since the poset $\mathfrak M$ satisfies the descending chain condition, every subposet of $\mathfrak M$ that is a lattice has a unique least element.  To show, say, that $W_\Psi$ is a minimal element of $\mathfrak N(\Psi),$  suppose $Y\preceq W_\Psi$ and $Y\in\mathfrak N(\Psi).$  Then Remark \ref{cprepr.1}(a) says that $Y\in\mathfrak M(c),$  whence $Y\in\mathfrak N(\Psi)\cap\mathfrak M(c)$ and we must have $W_\Psi\preceq Y.$  Thus $Y= W_\Psi.$\vskip.05in 

(b)  Let $W_{\alpha}=x_l\dots x_1.$ To prove  $W_{\alpha}\in\mathfrak P(c),$ proceed by induction on $l=\ell(W_\alpha).$ If $l=1,$ then ~\cite[Proposition 5.6(a)]{humph90} says that $\alpha=\alpha_{\rho(x_1)}$ is a simple root. Since $x_1$ is a sink of $(\Gamma, c),$ the statement holds. 

 If $l>1,$ suppose that, for all Coxeter  elements $d\in\mathcal W$, the statement holds for all $\beta\in\mathcal P(d)$ satisfying $\ell(W_{\beta})<l$.  Since $l>1$, Remark \ref{cprepr1.11}(b) says that $\rho({x_1})\alpha\in\mathcal P(\rho({x_1})c \rho({x_1})),$ and Remark \ref{cprepr1.3}(b) says that $W_{\rho({x_1})\alpha}=x_l\dots x_2$.  By the inductive hypothesis, $W_{\rho({x_1})\alpha}\in\mathfrak P({\rho({x_1})c \rho({x_1}})).$ Since $W_{\rho({x_1})\alpha}$ is  $\rho({x_1})c \rho({x_1})$-principal, the full subgraph $\Theta$ of $\Gamma$ determined by Supp$\,W_{\rho({x_1})\alpha}$ is connected. Now  ~\cite[Proposition 2.5]{kt08} says that $W_{\alpha}\in\mathfrak P(c)$ if the full subgraph $\Omega$ of $\Gamma$ determined by $\mathrm{Supp}\,W_{\alpha}$ is connected.     

Assume, to the contrary, that $\Omega$ is disconnected.  Then $x_1\not\in\mathrm{Supp}\,W_{\rho({x_1})\alpha}$ and no edge of $\Gamma$ joins $x_1$ and a vertex from $\mathrm{Supp}\,W_{\rho({x_1})}.$ Therefore $W_{\alpha}= W_{\rho({x_1})\alpha}x_1= x_1W_{\rho({x_1})\alpha}$ whence $W_{\rho({x_1})\alpha}\in\mathfrak M(c)$ by Remark \ref{cprepr.1}(a). Since $\rho(W_\alpha)\alpha<0,$ Lemma \ref{disconnected} says that either $x_1\in\mathfrak N(\alpha)$ or $W_{\rho({x_1})\alpha}\in\mathfrak N(\alpha).$  The former contradicts $\ell(W_\alpha)>1.$  The latter contradicts (a), for $W_{\rho({x_1})\alpha}\prec W_\alpha.$ Thus $\Omega$ is connected and $W_{\alpha}\in\mathfrak P(c).$

By Definition \ref{filter1}, $W_{\alpha}=W_{r,x}$ for some $r$ and $x.$ If $W_{\alpha}=W_{q,y}$ then $W_{r,x}= W_{q,y},$ so that $r=q$ and $x=y$ by Proposition \ref{filter2}(c).\vskip.05in

(c) Since $W_{\alpha}=x_l\dots x_1=W_{r,x}$ by (b), then $x_l=x$ by Proposition \ref{filter2}(d).  By (a),  if $0<j<l$ then $\rho(x_j)\dots \rho(x_1)\alpha>0,$ and $\rho(x)\rho(x_{l-1})\dots \rho(x_1)\alpha<0$.  By  ~\cite[Proposition 5.6(a)]{humph90}, $\alpha_{\rho(x)}$ is the only positive root negated by $\rho(x).$  Hence $\rho(x_{l-1})\dots \rho(x_1)\alpha=\alpha_{\rho(x)}$ and $\rho(W_{r,x})\alpha=-\alpha_{\rho(x)}.$  Since $\rho(X)^{-1}=\rho\left(X^T\right)$ for any $X,$ then $\alpha=\rho\left(W_{r,x}^T\right)\left(-\alpha_{\rho(x)}\right).$ If $0<i< r,$ then $W_{i,x}\prec W_{r,x}$ by Proposition \ref{filter2}(b).  Now the remaining inequality follows from (a).\vskip.05in 

(d) If $\beta\in\mathcal P(c)$ and $\nu({\alpha})= \nu({\beta}),$ then  (b) says that $W_{\alpha}= W_{\beta}=W_{r,x}$ for uniquely determined $r>0, \,x\in \Gamma_0.$ By (c), $\alpha=\beta=\rho\left(W_{r,x}^T\right)\left(-\alpha_{\rho(x)}\right).$
\end{proof}

The statement of Theorem \ref{cprepr2}(d) is not true for the map $\xi:\mathfrak f(c)\to\mathfrak M(c)$.  

\begin{defn}\label{size}  If $\alpha\in\mathcal P(c)$ satisfies $W_\alpha= W_{r,x},$ for some $r>0,\ x\in \Gamma_0,$ we say that $\alpha$ is a $c$-preprojective root of {\em size} $r$.    Denote by $\mathcal P(c,r)$ the set of $c$-preprojective roots of size $r.$
\end{defn}
Note that $\mathcal P(c,1)$ is the set of $c$-projective roots, for if a multiplicity-free word $X\in\mathfrak M(c)$ negates a root $\alpha>0,$ then $W_\alpha$ must be multiplicity-free because $W_\alpha\preceq X$ by Theorem \ref{cprepr2}(a). By Theorem \ref{cprepr2}(b), $\mathcal P(c,q)\cap\mathcal P(c,r)=\emptyset$ if $q\ne r.$

\begin{rmk}\label{cprepr2.5} Let $\alpha\in\mathcal P(c,r)$ so that $W_\alpha=W_{r,x},\,x\in\Gamma_0,$ and set $s=\rho(x).$  If $X=UW_\alpha$ and $x\not\in\mathrm{Supp}\,U,$ then $X\in\mathfrak N(\alpha).$

Indeed, Theorem \ref{cprepr2}(c) says that $\rho(X)\alpha=\rho(U)\left[\rho(W_\alpha)\alpha\right]=\rho(U)(-\alpha_s)<0,$ for the $s$-coordinate of $\rho(X)\alpha$ is $-1.$
\end{rmk}

 The following definition makes use of Definition \ref{independent.1} and  of Theorem \ref{cprepr2}(b) saying that $W_{\alpha}\in\mathfrak P(c)$ for all $\alpha\in\mathcal P(c).$
 
\begin{defn}\label{independent}  A finite subset $\{\alpha_1,\dots, \alpha_m\}$ of $\mathcal P(c)$ is {\em independent} if $\{W_{\alpha_1},\dots, W_{\alpha_m}\}$  is an independent subset of $\mathfrak P(c).$  Denote by $\mathfrak i(c)$ the subset of $\mathfrak f(c)$ consisting of all independent subsets of $\mathcal P(c),$ and denote by $\xi:\mathfrak i(c)\to\mathfrak M(c)$ the restriction of the map $\xi:\mathfrak f(c)\to\mathfrak M(c).$
\end{defn}

\begin{prop}\label{independent1}  
\begin{itemize}
\item[(a)] If $\Psi=\{\alpha_1,\dots, \alpha_m\},\,m>0,$ is an independent subset of  $\mathcal P(c),$  then $W_\Psi= W_{\alpha_1}\vee\dots\vee W_{\alpha_m}.$
\item[(b)] The map $\xi:\mathfrak i(c)\to\mathfrak M(c)$ given by $\xi(\Psi)= W_\Psi$ is injective.
\end{itemize}
\end{prop}

\begin{proof} (a)   We may assume $m>1.$ Since $\mathfrak N(\Psi)\cap\,\mathfrak M(c)=\overset{m}{\underset{i=1}\cap} \left[\mathfrak N(\alpha_i)\cap\,\mathfrak M(c)\right],$ Theorem \ref{cprepr2}(a) says that $W_{\alpha_i}\preceq W_\Psi$ for all $i.$  It suffices to show that $X=W_{\alpha_1}\vee\dots\vee W_{\alpha_m}$ belongs to $\mathfrak N(\alpha_i)$ for all $i,$ for then $X\in\mathfrak N(\Psi)\cap\,\mathfrak M(c)$ and $X\preceq W_\Psi,$ which forces $X= W_\Psi.$ Set $Y=W_{\alpha_1}\vee\dots\vee W_{\alpha_{i-1}}\vee W_{\alpha_{i+1}}\vee\dots\vee W_{\alpha_m}.$

Using parts (b)--(d) of Theorem \ref{lattice}, we get $W_{\alpha_{i}}= V( W_{\alpha_{i}}\wedge Y),$ $Y= W( W_{\alpha_{i}}\wedge Y),$ and  $X= VW\left(W_{\alpha_{i}}\wedge Y\right),$ where Supp$\,V\cap\,\mathrm{Supp}\, W=\emptyset$ and no edge of $\Gamma$ joins a vertex from Supp$\,V$ and a vertex from Supp$\,W.$  Since $\Psi$ is independent, Theorem \ref{cprepr2}(a) says that $W_{\alpha_{i}}\not\preceq Y$ whence $W_{\alpha_{i}}\not\preceq W_{\alpha_{i}}\wedge Y$  so that  $\rho(W_{\alpha_{i}}\wedge  Y)\alpha_i>0$  and $\rho(W)\left[\rho(W_{\alpha_{i}}\wedge  Y)\alpha_i\right]=\rho(Y)\alpha_i>0.$ Since  $\rho(V)\left[\rho(W_{\alpha_{i}}\wedge  Y)\alpha_i\right]<0,$   Lemma \ref{disconnected} says that $\rho(X)\alpha_i<0.$  Hence $X\in\mathfrak N(\alpha_i).$\vskip.05in

(b) Suppose $\Theta\in\mathfrak i(c)$ and $W_\Psi= W_\Theta,$ where $\Psi$ is from (a) and $\Theta=\{\beta_1,\dots, \beta_q\}.$ Then $q>0$ and $W_\Theta= W_{\beta_1}\vee\dots\vee W_{\beta_q}$ by (a). By Proposition \ref{independent.2}, $m=q$ and there exists a reindexing so that  $W_{\alpha_i}= W_{\beta_i}$ for all $i.$  By Theorem \ref{cprepr2}(d), $\alpha_i=\beta_i$ for all $i,$ that is $\Psi=\Theta.$
\end{proof}

We  define the values of the bilinear form $B$ on the paths in $(\Gamma,c)$ by setting $B(e_x)=1$ if $e_x$ is the trivial path at $x\in\Gamma_0,$  and 
$$B(p)=(-1)^tB\left(\alpha_{\rho(e(a_t))},\alpha_{\rho(s(a_t))}\right)\dots B\left(\alpha_{\rho(e(a_1))},\alpha_{\rho(s(a_1))}\right)$$ if $p=a_t\dots a_1$ is a nontrivial path.  It is straightforward that $B(p)>0$ for all paths $p,$  and  that $B(qp)=B(q)B(p)$ whenever $p$ and $q$ are paths satisfying $s(q)=e(p).$ \vskip.05in

The next statement gives a description of the $c$-projective roots.

\begin{thm}\label{cproj}  
Let $s\in S$  and denote by $x$ the unique vertex of $\Gamma$ satisfying $s=\rho(x).$ 
\begin{itemize}
\item[(a)]       The root $\pi_{s}(c)=\rho\left(W_{1,x}^T\right)(-\alpha_s)$  is $c$-projective  with $W_{\pi_{s}(c)}=W_{1,x}.$ The map $S\to\mathcal P(c,1)$ given by $s\mapsto\pi_{s}(c)$ is bijective, so that $\mathcal P(c,1)=\{\pi_{s}(c)\,|\,s\in S\}.$
\item[(b)] Let  $t=\rho(y)\in S,\,y\in\Gamma_0.$  If $x\not\le y,$ the $t$-coordinate of $\pi_{s}(c)$ is $0.$  If $x\le y,$ the $t$-coordinate of $\pi_{s}(c)$ is
$\underset{p}\sum B(p),$  where $p$ runs through all paths from $x$ to $y$ in $(\Gamma,c).$
  Hence
 \centerline{$\pi_{s}(c)=\underset{y\in \Gamma_0,\,x\le y}\sum\left(\underset{p:x\to y}\sum B(p)\right)\alpha_{\rho(y)}.$}\vskip.05in  
\item[(c)]  $\pi_{s}(c)=-c^{-1}\pi_{s}(c^{-1})$ and $\pi_{s}(c^{-1})=-c\pi_{s}(c).$  Therefore, the linear operators $-c$ and $-c^{-1}$ on $\mathbb V$ induce mutually inverse bijections $-c:\mathcal P(c,1)\to \mathcal P(c^{-1},1)$ and $-c^{-1}: \mathcal P(c^{-1},1)\to \mathcal P(c,1)$.
\end{itemize}
\end{thm}

\begin{proof} Let $W_{1,x}=x_l\dots x_1.$

(a)  We have $\rho\left(W_{1,x}\right)\pi_{s}(c)=\rho\left(W_{1,x}\right)\left[\rho\left(W_{1,x}^T\right)(-\alpha_s)\right]=-\alpha_s<0.$ Since  $x_l=x$ by  Proposition \ref{filter2}(d),  the $s$-coordinate of $\pi_s(c)$ is  $1.$  Therefore $\pi_s(c)\in\mathcal P(c,1)$ and  $x\in\mathrm{Supp}\,W_{\pi_s(c)}$ so that $\mathrm{Supp}\,W_{1,x}=\langle x\rangle\subset\mathrm{Supp}\,W_{\pi_s(c)}$ because $\mathrm{Supp}\,W_{\pi_s(c)}$ is a filter of $(\Gamma_0,c)$  by Proposition \ref{filter}(a).  By Theorem \ref{cprepr2}(a), $W_{1,x}= VW_{\pi_{s}(c)}$  whence $V=1$ because $W_{1,x}$ is multiplicity-free.  Thus $W_{1,x}=  W_{\pi_s(c)}.$

To show the map $s\mapsto  \pi_s(c)$ is injective, let $t\in S$ satisfy $t=\rho(y)$  and $\pi_s(c)=\pi_t(c).$ Then $W_{\pi_{s}(c)}= W_{\pi_{t}(c)}$ by Theorem \ref{cprepr2}(a) so that $W_{1,x}= W_{1,y}$ by what we have just proved. By Proposition \ref{filter2}(c),  $x= y$ so that $s=t.$ 

To show the map is surjective, let $\alpha\in\mathcal P(c,1).$ By Theorem \ref{cprepr2}(a), $W_\alpha$ is multiplicity-free. Then Theorem \ref{cprepr2}(b) says that $W_\alpha=W_{1,x},$ for some $x\in \Gamma_0,$  and Theorem \ref{cprepr2}(c) gives $\alpha=\rho(W_{1,x}^T)\left(-\alpha_{\rho(x)}\right)=\pi_{\rho(x)}(c).$\vskip.05in

(b) By (a), $W_{\pi_s(c)}=W_{1,x},$ and we proceed by induction on $l.$ 

If $l=1,$  then $\pi_s(c)=\alpha_s,$ so that the $t$-coordinate is 1 if $t=s,$ and it is 0 if $t\ne s.$  Since $x$ is a sink, $x\le y$ implies $x=y.$ Since $e_x$ is the only path from $x$ to $x,$ and $B(e_x)=1$ by definition, the statement holds.

If $l>1,$ suppose that, for all Coxeter elements $d\in\mathcal W,$ the statement holds for all $u\in S$ satisfying $\ell\left(W_{\pi_u(d)}\right)<l.$ Set $Y=x_l\dots x_2$ and $\beta=\rho(x_1)\pi_s(c).$ Remark \ref{cprepr1.11}(b) says that $d=x_1\cdot c=\rho(x_1)c\rho(x_1),$ and Remark \ref{cprepr.1}(a) says that $Y\in\mathfrak M(d).$ By Remark \ref{cprepr1.3}(b),  $\beta>0$ and $Y$ is a minimal element of $\mathfrak N(\beta).$ Therefore $\beta\in\mathcal P(d,1)$ and  $Y=W_{\beta}$  according to Theorem \ref{cprepr2}(a). By (a), $\beta=\pi_{u}(d)$ for some $u\in S,$ and $Y=W_{1,z}\in\mathfrak P(d)$ where $z\in\Gamma_0$  satisfies $u=\rho(z).$  By Proposition \ref{filter2}(d), $z=x_l=x$ so that $\beta=\pi_{s}(d)$ and $Y=W_{1,x}.$  In particular, $\{x_2,\dots,x_l\}=\mathrm{Supp}\,Y$ is the principal filter of $(\Gamma_0,d)$ generated  by $x.$  Since $\ell(Y)=l-1,$  the statement holds for $\pi_{s}(d),$ so
\newline\centerline{$\rho(x_1)\pi_{s}(c)=\pi_{s}(d)=\overset{l}{\underset{j=2}\sum}\left(\underset{q:x\to x_j}\sum B(q)\right)\alpha_{\rho({x_j})}$} where, for each $j,$ $q$ runs through all paths from $x$ to $x_j$ in the quiver $(\Gamma,d).$   Applying $\rho(x_1)$ to both sides of the above equality, we get  
\newline\centerline{$\pi_s(c)=\overset{l}{\underset{j=2}\sum}\left(\underset{q:x\to x_j}\sum B(q)\right)\left[\alpha_{\rho({x_j})}-B\left(\alpha_{\rho(x_1)},\alpha_{\rho(x_j)}\right)\alpha_{\rho(x_1)}\right]$} \newline\centerline{$=\overset{l}{\underset{j=2}\sum}\left(\underset{q:x\to x_j}\sum B(q)\right)\alpha_{\rho(x_j)}+\overset{l}{\underset{j=2}\sum}\left(\underset{q:x\to x_j}\sum -B\left(\alpha_{\rho(x_1)},\alpha_{\rho(x_j)}\right)B(q)\right)\alpha_{\rho(x_1)}$}\newline\centerline{$=\overset{l}{\underset{j=1}\sum}\left(\underset{p:x\to x_j}\sum B(p)\right)\alpha_{\rho(x_j)}.$}

\vskip.05in\noindent Note that since $x_1$ is a sink in $(\Gamma,c)$ and a source in  $(\Gamma,d),$  the paths $x\to x_j$ are the same in both quivers  for $j>1.$ And each path $p:x\to x_1$ in $(\Gamma,c)$ satisfies $p=aq$ for a unique path $q:x\to x_j$ with $j>1$ and a unique arrow $a:x_j\to x_1$ in $(\Gamma,c)$ (remember, $x\ne x_1$ because $l>1$).  Here $B(p)=B(a)B(q)=-B\left(\alpha_{\rho(x_1)},\alpha_{\rho(x_j)}\right)B(q).$\vskip.05in
 
(c) The ideal $(x)$ of $(\Gamma_0,c)$ generated by $x$ is the filter of the poset  $\left(\Gamma_0,c^{-1}\right)$ generated by $x.$   By  Proposition \ref{filter}(a), there exists a unique  multiplicity-free $Z\in\mathfrak M(c^{-1})$ for which $(x)=\mathrm{Supp}\,Z.$  Applying (a) to $c^{-1}$ instead of $c,$ in view of Definition \ref{filter1} we get that $\pi_s(c^{-1})=\rho(Z^T)(-\alpha_s)$ is a $c^{-1}$-projective root. Using  Proposition \ref{filter2}(b) and  Remark \ref{cprepr1.11}(a), we obtain $K^T= VZ,$ for some $V.$ Then $K= Z^TV^T$ and Remark \ref{cprepr.1} says that $V^T\in\mathfrak M(c).$  Setting $U=x_{l-1} \dots x_1,$ we note that $U\in\mathfrak M(c)$ and Supp$\,U\subset\mathrm{Supp}\, V^T,$ for Supp$\,Z\cap\,\mathrm{Supp}\,W_{1,x}=(x)\cap\langle x\rangle=\{x\}.$  Since $U$ and $V^T$ are mulltiplicity-free, Proposition \ref{filter}(b) says that $U\preceq V^T$ so that $V^T= Y^TU,$ for some $Y.$ Then $V= U^TY$ and $K^T= U^TYZ$ where Supp$\,Y=\Gamma_0\setminus[\langle x\rangle\cup(x)].$

Let $Y=y_m\dots y_1$ and $0<i\le m.$ Then no edge of $\Gamma$ joins $x$ and $y_i,$  for  there is no arrow $x\to y_i$ in $(\Gamma, c)$ because $y_i\not\in\langle x\rangle,$ and no arrow $y_i\to x$ in $(\Gamma, c)$ because $y_i\not\in(x).$  Hence $\rho({y_i})(\alpha_s)=\alpha_s$ because $B\left(\alpha_s,\alpha_{\rho(y_i)}\right)=0.$ It follows that $\rho(Y)(\alpha_s)=\alpha_s.$

In view of Remark \ref{cprepr1.11}(a), we have
$$-c^{-1}\pi_s\left(c^{-1}\right)=-\rho(K^T)\rho(Z^T)(-\alpha_{s})=-\rho(U^T)\rho(Y)\rho(Z)\rho(Z^T)(-\alpha_{s})=-\rho(U^T)\rho(Y)(-\alpha_{s})=$$
$$-\rho(U^T)(-\alpha_{s})=\rho(U^T)s(-\alpha_{s})=\rho(U^T)\rho(x)(-\alpha_{s})=\rho(W_{1,x}^T)(-\alpha_{s})=\pi_s(c)$$  because $x_l=x.$ The rest is clear.
\end{proof}

By Theorem \ref{cproj}(a), there is a bijection between the vertices of $\Gamma$ and the $c$-projective roots, which is similar to the well-known bijection between the vertices of a quiver and the nonisomorphic indecomposable projective representations of the quiver.  The bijection of Theorem \ref{cproj}(c) is the analog of the bijection between the indecomposable projective and indecomposable injective representations.

\begin{defn}\label{cproj1}  The root $\pi_s(c)$  of Theorem \ref{cproj}(a) is the $c$-projective root {\em associated with} $s\in S.$
\end{defn}

Parts (a) and (c) of the following statement are analogs of the well-known properties of indecomposable preprojective representations of a quiver.

\begin{prop}\label{cprepr3}  Let  $\alpha\in\mathcal P(c,r)$ and let $X_r\dots X_1$ be the canonical form of $W_{\alpha}= W_{r,x},\,r>0,\, x\in\Gamma_0.$  Set $s=\rho(x).$ Let $i$ be an  integer satisfying $0< i<r.$
\begin{itemize}
\item[(a)] $c^i\alpha>0$ and $c^r\alpha<0.$ Hence $c^i\alpha\in\mathcal P(c).$
\item[(b)]  $c^i\alpha=\rho(X_i\dots X_1)\alpha$ and $X_r\dots X_{i+1}$ is the canonical form of $W_{c^i\alpha}.$ 
\item[(c)]   $c^{-i}\pi_s(c)>0$ and  $\alpha=c^{-r+1}\pi_s(c),$ where $\pi_s(c)$ is the $c$-projective root associated with $s.$
\end{itemize}
\end{prop}

\begin{proof} Set $K=v_n\dots v_1$ so that $c=\rho(K).$\vskip.03in

(a) If $j>0$ and $c^j\alpha<0$, then Remark \ref{cprepr1.11}(a) and Theorem \ref{cprepr2}(a) say that $W_{r,x}\preceq K^j$.  By Proposition \ref{filter2}(b), $r\le j$ and  $K^r= UW_{\alpha}$ where $x\not\in\mathrm{Supp}\,U$.  Hence $c^i\alpha>0$ and Remark \ref{cprepr2.5} says that $K^r\in\mathfrak N(\alpha),$ whence  $c^r\alpha=\rho(K^r)\alpha<0$.\vskip.03in

(b) There is nothing to prove if $r=1,$ so let $r>1$ and suppose the statement holds for all $\beta\in\mathcal P(c,r-1).$ We show first that $c\alpha=\rho(X_1)\alpha$ and   $X_r\dots X_2$ is the canonical form of $W_{c\alpha},$ i.e., that the statement holds for $i=1.$

By (a), $c^{r-1}(c\alpha)<0$ and $c^i(c\alpha)>0$ if $0\le i<r-1.$ Hence $c\alpha\in\mathcal P(c,r-1)$ and $W_{c\alpha}=W_{r-1,y}$ for some $y\in\Gamma_0.$  By Propositon \ref{filter2}(b), $W_{r-1,x}K$ and $W_{r-1,y}K$ belong to $\mathfrak M(c).$ Since $\rho(W_{c\alpha})c\alpha=\rho(W_{r-1,y}K)\alpha<0$, Theorem \ref{cprepr2}(a) says that $W_{r,x}\preceq W_{r-1,y}K.$ By Proposition \ref{filter2}(b),  $x\in\langle y\rangle$ and $W_{r,x}\preceq W_{r-1,x}K.$ Therefore $y\le x$ and  $W_{r-1,x}K= UW_{r,x}$ where  $x\not\in\mathrm{Supp}\,U$.  By Remark \ref{cprepr2.5}, $W_{r-1,x}K\in\mathfrak N(\alpha)$ whence $W_{r-1,x}\in\mathfrak N(c\alpha)$ so that  $W_{r-1,y}\preceq W_{r-1,x}$.  By Proposition \ref{filter2}(b), $x\le y$.  Thus $y=x$ and $W_{c\alpha}=W_{r-1,x}= X_r\dots X_2=W_{\rho(X_1)\alpha}$ in view of Remark \ref{cprepr1.3}(b). Since $W_{c\alpha}= W_{\rho(X_1)\alpha},$ then $c\alpha=\rho(X_1)\alpha$ by Theorem \ref{cprepr2}(d).  We have proved that the statement holds for $i=1.$

The inductive hypothesis says that  if $0<j<r-1,$ then $$c^j(c\alpha)=c^j\left[\rho(X_1)\alpha\right]=\rho\left(X_{j+1}\dots X_2\right)\left[\rho(X_1)\alpha\right]$$ and $X_r\dots X_{j+2}$ is the canonical form of $W_{c^j\left[\rho\left(X_1\right)\alpha\right]}=W_{c^{j+1}\alpha}$  Setting $i=j+1,$ we see that the statement holds if $1<i<r.$\vskip.03in

(c) By (a) and (b), $c^{r-1}\alpha>0$  and $W_{c^{r-1}\alpha}= X_r=W_{1,x}.$ Since $W_{\pi_s(c)}=W_{1,x}$ by Theorem \ref{cproj}(a), then  $W_{c^{r-1}\alpha}= W_{\pi_s(c)}$ whence $c^{r-1}\alpha=\pi_s(c)$ by Theorem \ref{cprepr2}(d).  Therefore $\alpha=c^{-r+1}\pi_s(c).$ If $0< i< r,$ then $c^{-i}\pi_s(c)=c^{-i+r-1}\alpha>0$ by (a).
\end{proof}

We show how to construct the $c$-preprojective roots of size $r.$

\begin{prop}\label{cprepr3.1}  Let $r>0$ be an integer and let $s=\rho(x),\,x\in\Gamma_0.$ If $\rho\left(W_{i,x}^T\right)(-\alpha_{s})>0$ when $1<i\le r,$ then $\alpha=\rho\left(W_{r,x}^T\right)(-\alpha_{s})\in\mathcal P(c,r)$ and $W_\alpha= W_{r,x}$.
\end{prop}

\begin{proof}  It is straightforward that  $\rho(W_{r,x})\alpha=-\alpha_s,$ so $\alpha\in\mathcal P(c)$.  We have to prove that $W_\alpha= W_{r,x}.$ Let $X_r\dots X_1$ be the canonical form of $W_{r,x}.$ 

By Theorem \ref{cprepr2}(b), $W_\alpha= W_{q,y}$ for some $q>0,\, y\in\Gamma_0.$  Setting $t=\rho(y),$   we note that $W_{q,y}\preceq W_{r,x}$ by Theorem \ref{cprepr2}(a), so Proposition \ref{filter2}(b) says that $q\le r$ and $y\in\mathrm{Supp}\, X_q.$ Set $Y=X_q\dots X_1.$ Then $W_{q,y}\preceq Y$  so that $Y= UW_{q,y},$ where $y\not\in\mathrm{Supp}\,U$ because each $X_j$ is multiplicity-free. By Remark \ref{cprepr2.5}, 
$$\rho(Y)\alpha=\left[\rho(X_q\dots X_1)\rho(X_1^T\dots X_r^T)\right](-\alpha_{s})<0$$ whence $q\ge r$ because, by assumption, 
$$\rho(X_{q+1}^T\dots X_r^T)(-\alpha_{s})=\rho\left(W_{r-q,x}^T\right)(-\alpha_{s})>0$$ if $0< q<r.$ Therefore $q=r$ so that $Y=W_{r,x}.$ In view of Theorem \ref{cprepr2}(c),
$$-\alpha_s=\rho\left(W_{r,x}\right)\alpha=\rho(U)\left[\rho\left(W_{r,y}\right)\alpha\right]=\rho(U)(-\alpha_t).$$
 Since $y\not\in\mathrm{Supp}\,U,$ we must have $t=s$ whence $y=x.$ Thus $W_\alpha= W_{r,x}.$
\end{proof}

\begin{thm}\label{cprepr3.3}  The following are equivalent for a root $\alpha>0$ and an integer $r>0$.
\begin{itemize}
\item[(a)]  $\alpha\in\mathcal P(c,r)$.
\item[(b)] $c^i\alpha>0$ whenever $0<i<r,$ and $c^r\alpha<0.$ 
\item[(c)]  There is an $s\in S$ for which $c^{-i}\pi_{s}(c)>0$ if $0< i< r,$ and $\alpha=c^{-r+1}\pi_{s}(c);$ here $\pi_{s}(c)$ is the $c$-projective root associated with $s\in S.$
\item[(d)]  There is an $x\in \Gamma_0$ for which $\rho\left(W_{i,x}^T\right)(-\alpha_{s})>0$ if $1< i\le r,$ and $\alpha=\rho\left(W_{r,x}^T\right)(-\alpha_{s}),$ where $s=\rho(x).$
\end{itemize}
If $\alpha$ satisfies any of the conditions (a)-(d), the elements $s$ in (c) and $x$ in (d) are uniquely determined.
\end{thm}
\begin{proof} (a)$\implies$(b)  and (a)$\implies$(c)  These are parts (a) and (c) of Proposition \ref{cprepr3}. 

(d)$\implies$(a)  This is Proposition \ref{cprepr3.1}.

(c)$\implies$(b) This  is an immediate consequence of Theorem \ref{cproj}(c).

(b)$\implies$(d) By definition, $\alpha\in\mathcal P(c)$ so that Theorem \ref{cprepr2}(b) says that $W_\alpha=W_{q,x},$ for some $q>0,x\in\Gamma_0.$ By Theorem \ref{cprepr2}(c), $\rho(W_{i,x}^T)(-\alpha_{s})>0$ if $0<i\le q,$ and $\alpha=\rho(W_{q,x}^T)(-\alpha_{s}).$  By Proposition \ref{cprepr3}(a), $c^i\alpha>0$ if $0<i<q,$ and $c^q\alpha<0.$  Comparing the latter with the assumption, we get $q=r.$ 

The uniqueness is an immediate consequence of our prior results.
\end{proof}

The final statement of  this section follows immediately from Theorems \ref{cprepr3.3} and \ref{cproj}(c).

\begin{cor}\label{cprepr3.5}
\begin{itemize}
\item[(a)]  A root $\alpha>0$ is $c$-preprojective if and only if $c^r\alpha<0$, for some integer $r>0$.  The root $\alpha$ is $c$-projective if and only if $c\alpha<0$.
\item[(b)] For all $r>0,$ 
\newline$\mathcal P(c,r)=\{\rho\left(W_{r,x}^T\right)(-\alpha_{\rho(x)})\,|\, x\in\Gamma_0,\, \rho\left(W_{i,x}^T\right)(-\alpha_{\rho(x)})>0\ \mathrm{if}\  1< i\le r\}\newline\phantom{xxxxx}=\{c^{-r+1}\pi_{\rho(x)}\,|\, x\in\Gamma_0,\ c^{-i}\pi_{\rho(x)}>0\ \mathrm{if}\ 0< i< r\}.$\end{itemize}
\end{cor}

\section{Preprojective roots and finite Coxeter groups}

\begin{thm}\label{cprepr5} For  a Coxeter element $c\in\mathcal W$,  the following are equivalent.
\begin{itemize}
\item[(a)]  The group $\mathcal W$ is finite.
\item[(b)]  The set $\mathcal P(c)$ is finite.
\item[(c)] All positive roots are $c$-preprojective.
\item[(d)] All simple roots are $c$-preprojective.
\end{itemize}
\end{thm}
\begin{proof}  (a)$\implies$(b) and (c)$\implies$(d) are trivial.

(b)$\implies$(a)  By ~\cite[Proposition 5.6(b)]{humph90}, the length of the element $c^t,$ for some $t>0,$  is the number of positive roots $\alpha$ satisfying $c^t\alpha<0.$  Since a power of $c$ is $c$-admissible by Remark \ref{cprepr1.11}(a), each such $\alpha$ belongs to $\mathcal P(c)$  by Remark \ref{cprepr1.9}. Since $\mathcal P(c)$ is a finite set,  the lengths of the powers $c^t,\ t>0$, are bounded by $|\mathcal P(c)|.$ Hence there are only finitely many distinct elements among such powers, so that the order of $c$ is finite. By ~\cite[Theorem 4.1]{h}, $\mathcal W$ is finite.

(a)$\implies$(c) The unique element $w_0\in\mathcal W$ of maximal length satisfies $w_0\alpha<0,$ for all roots $\alpha>0,$  according to  ~\cite[Proposition 5.6 and Exercises 1, 2, p. 115]{humph90}. By ~\cite[Theorem 3]{spey08},  $w_0$ is $c$-admissible.  Hence all positive roots are $c$-preprojective.

(d)$\implies$(a)  By assumption, $\Psi=\{\alpha_s\,|\,s\in S\}\subset\mathcal P(c).$  By Theorem \ref{cprepr2}(a), the word $W_\Psi\in\mathfrak M(c)$ satisfies  $\rho(W_\Psi)\alpha_s<0,$   for all $s\in S,$ whence  $\rho(W_\Psi)\alpha<0$  for all $\alpha\in\Phi^+.$ Therefore the group $\mathcal W$ must be finite.
\end{proof}

\begin{prop}\label{cproj2}  For  a Coxeter element $c\in\mathcal W$,  the following are equivalent.
\begin{itemize}
\item[(a)] $|\mathcal W|=2.$
\item[(b)] All positive roots are $c$-projective.
\item[(c)] All simple roots are $c$-projective.
\end{itemize}
\end{prop}
\begin{proof}  (a)$\implies$(b) and (b)$\implies$(c) are clear. 

(c)$\implies$(a) For each  $x\in\Gamma_0,$ set $s=\rho(x).$   Since $\alpha_s$ is $c$-projective, Theorem \ref{cproj}(b) says that $\pi_s(c)=\alpha_s,$ whence $e_x$ is the only path in $(\Gamma,c)$ starting at $x;$ in particular, no arrow starts at $x.$  Hence $\Gamma_1=\emptyset.$  Since we assume $\mathcal W$ irreducible, $\Gamma$ is connected and must consist of the single vertex $x.$ Thus $S=\{s\}$  and $|\mathcal W|=2.$
\end{proof}

\section{Reduced $c$-admissible words}

We quote ~\cite[formula (1.19), p. 26, and Definition 3.1.1(ii)]{bb2005}.
\begin{defn}\label{cprepr3.6} The set of {\em right associated reflections} for $w\in\mathcal W$ is ${T}_R(w)=\{t\in {T}\,|\, \ell(wt)<\ell(w)\}.$  For $u,v\in\mathcal W,$  set $u\le_L v$ if and only if $v=s_k s_{k-1}\dots s_1u$, for some $s_i\in {S}$ satisfying $\ell(s_i s_{i-1}\dots s_1u)=\ell(u)+i,\ 0\le i\le k$.  The binary relation $\le_L$ is the {\em left weak order} on $\mathcal W.$
 \end{defn}  

\begin{prop}\label{cprepr3.7} A root $\alpha>0$ is $c$-preprojective (resp. $c$-projective) if and only if $s_\alpha\in{T}_R(w),$ where $w=\rho(X)$ for some $X\in\mathfrak M(c)$ (resp. multiplicity-free $X\in\mathfrak M(c)$).
\end{prop} 
\begin{proof}  By ~\cite[Proposition 5.7]{humph90} $\ell(ws_\alpha)<\ell(w)$ if and only if $w(\alpha)<0.$
\end{proof}

We relate the partial order $\preceq$ on $\mathfrak M(c)$ to  the left weak order on the set of $c$-admissible elements of $\mathcal W.$

\begin{lem}\label{cprepr4}   Let  $X,Y\in\mathfrak M.$
\begin{itemize}
\item[(a)] If $X\preceq Y$ and the word $\rho(Y)$ is reduced, then the word $\rho(X)$ is reduced and $\rho(X)\underset L\le\rho(Y)$.
\item[(b)]  For any $\alpha\in\mathcal P(c)$ and $Y\in\mathfrak M(c),$ if  $\rho(W_\alpha)\underset L\le\rho(Y)$ then $W_\alpha\preceq Y$.
\end{itemize}
\end{lem}
\begin{proof} (a)  The proof is straightforward.\vskip.05in

(b) Since $\rho(W_\alpha)\underset L\le\rho(Y),$  ~\cite[Proposition 3.13]{bb2005} says that ${T}_R(\rho(W_\alpha))\subset{T}_R(\rho(Y))$.  Since $\rho(W_\alpha)\alpha<0$, ~\cite[Proposition 5.7]{humph90} says that $\ell(\rho\left(W_\alpha)s_{\alpha}\right)<\ell(\rho(W_\alpha))$ whence $s_{\alpha}\in{T}_R(\rho(W_\alpha)).$ Then $s_{\alpha}\in{T}_R(\rho(Y))$ whence $\ell(\rho(Y)s_{\alpha})<\ell(\rho(Y))$ so that $\rho(Y)\alpha<0$.  By Theorem \ref{cprepr2}(a),  $W_\alpha\preceq Y$.
\end{proof}

\begin{thm}\label{cprepr4.5} Let $\Psi=\{\alpha_1,\dots,\alpha_m\}$ be a finite subset of $\mathcal P(c).$
\begin{itemize}
\item[(a)] The word $\rho\left(W_\Psi\right)$ is reduced.  In particular, for all $\alpha\in\mathcal P(c),$ the word $\rho(W_\alpha)$ is reduced.
\item[(b)]  If $\alpha, \beta\in\mathcal P(c),$ then $W_\alpha\preceq W_\beta$ if and only if $\rho(W_\alpha)\underset L\le\rho(W_\beta).$
\item[(c)]  Assume $W_\Psi= W_{\alpha_1}\vee\dots\vee W_{\alpha_m}.$ For all $Y\in\mathfrak M(c),$  if $\rho(W_\Psi)\underset L\le\rho(Y)$ then $W_\Psi\preceq Y$.
\item[(d)] If $\Psi,\Theta\in\mathfrak i(c),$ then $W_\Psi\preceq W_\Theta$ if and only if $\rho(W_\Psi)\underset L\le\rho(W_\Theta).$
\end{itemize}
\end{thm}
\begin{proof} (a)  If $\mathcal W$ is infinite, the word $\rho(X)$ is reduced for all $X\in\mathfrak M(c)$ by ~\cite[Theorem 2]{spey08}.  Since $W_\Psi\in\mathfrak M(c),$ the statement holds.

Suppose $\mathcal W$ is finite.  By ~\cite[Theorem 3]{spey08},  the element $w_0$ of maximal length satisfies $w_0=\rho(U),$ where $U\in\mathfrak M(c)$ and the word $\rho(U)$ is reduced.   Since $U\in\mathfrak N(\Psi)$  then $W_\Psi\preceq U.$  By  Lemma \ref{cprepr4}(a), the word $\rho\left(W_{\Psi}\right)$ is reduced. \vskip.05in

(b) This follows from (a) and Lemma \ref{cprepr4}.

\vskip.05in(c) For each $i$ we have $W_{\alpha_i}\preceq W_\Psi,$ so Lemma \ref{cprepr4}(a) gives $\rho(W_{\alpha_i})\underset L\le\rho\left(W_\Psi\right)$ in view of (a).  Then $\rho(W_{\alpha_i})\underset L\le\rho(Y)$ by transitivity, so Lemma \ref{cprepr4}(b) gives $W_{\alpha_i}\preceq Y.$  Since $i$ is arbitrary, $W_\Psi\preceq Y$ by the property of join.

\vskip.05in(d)  The necessity follows from Lemma \ref{cprepr4}(a) because $\rho(W_\Theta)$ is reduced according to (a).  The sufficiency is a special case of (c).
\end{proof}

\begin{lem}\label{cprepr6}
Let  $X = x_l\dots x_1$ belong to ${\mathfrak P}(c),$ $l >
1,$ and set $Y = x_l \dots x_2$.
If $Y= W_\beta$ for some  $\beta\in\mathcal P({\rho({x_1})c\rho({x_1})})$  where $\beta\ne\alpha_{\rho(x_1)},$ then $\alpha=\rho({x_1})\beta\in\mathcal P(c)$ and $X=W_\alpha.$
\end{lem}

\begin{proof} Remarks \ref{cprepr.1}(a) and \ref{cprepr1.11}(b) say that $Y\in\mathfrak M\left({\rho(x_1)}c\rho({x_1})\right).$  Since $\beta>0$ and $\beta\ne\alpha_{\rho(x_1)}$, ~\cite[Proposition 5.6(a)]{humph90} says that $\alpha>0$.  By Theorem \ref{cprepr2}(c) and Proposition \ref{filter2}(d), $\beta=\rho\left(Y^T\right)(-\alpha_{\rho(x_l)})$ whence $\rho(X)\alpha=\rho(Y)\beta=-\alpha_{\rho(x_l)}<0.$  Therefore $\alpha\in\mathcal P(c)$ and  Theorem \ref{cprepr2}(a) says that $W_\alpha\preceq X$.

If $x_1\in\mathrm{Supp}\, W_\alpha,$ then $x_1\preceq W_\alpha$ by   Remark  \ref{cprepr.1}(b) because $x_1$ is a sink in $(\Gamma,c).$ Then $W_\alpha= Ux_1,$ for
some  $U\in\mathfrak M\left({\rho(x_1)}c\rho({x_1})\right),$  whence
$\rho(U)\beta=\rho(U)[\rho(x_1)\alpha]=\rho(W_\alpha)\alpha<0.$ By Theorem \ref{cprepr2}(a),  $Y\preceq U$ whence $X=Yx_1\preceq Ux_1= W_\alpha.$  Thus $X= W_\alpha$ as claimed.

If $x_1\not\in\mathrm{Supp}\, W_\alpha,$ then $x_1\not\in\mathrm{Supp}\, \alpha$ by Remark \ref{cprepr1.3}(a), and no edge of $\Gamma$ joins $x_1$ with a vertex of $\mathrm{Supp}\, W_\alpha$ because $x_1$ is a sink, while $\mathrm{Supp}\, W_\alpha$  is a filter of $(\Gamma_0,c)$ by Proposition \ref{filter}(a).  Hence $W_\alpha x_1= x_1W_\alpha\in\mathfrak M(c) $ so that $W_\alpha\in \mathfrak M\left({\rho(x_1)}c\rho({x_1})\right),$ and $W_\alpha x_1\preceq X=Yx_1$ by Proposition \ref{filter}(b).  Therefore $W_\alpha\preceq Y,$ and we have 
$$\rho(W_\alpha)\beta=\rho(W_\alpha)\left[\rho(x_1)\alpha\right]=\rho(x_1)[\rho(W_\alpha)\alpha]<0,$$
whence $Y\preceq W_\alpha$ so that $Y= W_\alpha.$  Thus $x_1\not\in\mathrm{Supp}\, Y=\mathrm{Supp}\,W_\alpha,$  and no edge of $\Gamma$ joins $x_1$ with a vertex of $\mathrm{Supp}\, Y$  as we noted above. Since Supp$\, X=\mathrm{Supp}\, Y\cup \{x_1\},$ the full subgraph of $\Gamma$ determined by $\mathrm{Supp}\, X$ is disconnected, which contradicts the assumption that $X\in{\mathfrak P}(c)$.
\end{proof}

\begin{defn}\label{cprepr11} An element $w\in\mathcal W$ has a  {\em $c$-admissible} (resp. {\em $c$-principal}) {\em reduced expression}  if  $w=\rho(X),$ where the word $\rho(X)$ is reduced and $X\in\mathfrak M(c)$ (resp. $X\in\mathfrak P(c)$).  For any  $\mathfrak A\subset \mathfrak M,$ denote by  Red$\,\mathfrak A$ the set of elements $Z\in \mathfrak A$ for which the word $\rho(Z)$  is reduced.
\end{defn}

We now characterize the words in Red$\,{\mathfrak P}(c).$ Recall that a map $f:P\to Q$ of posets is {\em order-preserving} if $x\le y$ implies $f(x)\le f(y),$ and $f$ is {\em order-reflecting} if $f(x)\le f(y)$ implies $x\le y.$  A map is {\em order-embedding} if it is both order-preserving and order-reflecting.

\begin{thm}\label{cprepr7}  
\begin{itemize}
\item[(a)] For any $X\in \mathfrak P(c),$ the word $\rho(X)$ is reduced if and only if  $X= W_\alpha$  for some $\alpha\in\mathcal P(c).$ Hence $\mathrm{Red}\,\mathfrak P(c)$ is the image of the map $\nu:\mathcal P(c)\to\mathfrak P(c),$ and the map $\nu:\mathcal P(c)\to\mathrm{Red}\,\mathfrak P(c)$ is bijective.
\item[(b)]  The map $\rho:\mathrm{Red}\,\mathfrak P(c)\to\mathcal W$ is order-embedding, hence, injective.  
\item[(c)]  The maps of (a) and (b) are bijections between $\mathcal P(c),$ $\mathrm{Red}\,\mathfrak P(c),$ and the set of elements of $\mathcal W$ having a $c$-principal reduced expression. If $\mathcal W$ is finite, $\mathcal P(c)=\Phi^+.$
\end{itemize}

\end{thm}
\begin{proof} (a) The  sufficiency is Theorem \ref{cprepr4.5}(a).  Let $X=x_l\dots x_1$ belong to $\mathfrak P(c)$ and suppose that $\rho(X)$ is reduced. We prove the necessity by induction on $l=\ell(X)$.

If $l=1$ then $X=x_1= W_{\alpha_{\rho(x_1)}}$ where $\alpha_{\rho(x_1)}\in\mathcal P(c)$  because $x_1$ is a sink.  Let  $l>1$ and  suppose that, for all Coxeter elements $d\in\mathcal W$, the statement  holds for all words in $\mathfrak P(d)$ of length $<l$.  By ~\cite[Proposition 4.6]{kp07}, see also ~\cite[Proposition 2.5]{kt08}, $Y=x_l\dots x_2$ belongs to $\mathfrak P\left({\rho(x_1)}c\rho({x_1})\right).$ Since $\rho(X)$ is reduced, $\rho(Y)$ is reduced by Lemma \ref{cprepr4}(a).   By the inductive hypothesis, $ Y= W_\beta$ for some $\beta\in\mathcal P\left({\rho(x_1)c\rho(x_1)}\right),$  so Theorem \ref{cprepr2}(c) says that $\beta=\rho\left(Y^T\right)\left(-\alpha_{\rho(x_l)}\right).$ In view of    ~\cite[Theorem 5.4]{humph90},
$$\rho(x_1)\beta=\rho(x_1)\rho\left(Y^T\right)\left(-\alpha_{\rho(x_l)}\right)=\rho\left(X^T\right)\left(-\alpha_{\rho(x_l)}\right)=\rho(x_1\dots x_{l-1})\alpha_{\rho(x_l)}>0$$ because the word  $\rho\left(X^T\right)$ is reduced.  Since   $\beta\ne\alpha_{\rho(x_1)}$  by ~\cite[Proposition 5.6(a)]{humph90}, Lemma \ref{cprepr6} says that  $X= W_\alpha$ for some $\alpha\in\mathcal P(c)$.

We have proved that the map $\nu:\mathcal P(c)\to\mathrm{Red}\,\mathfrak P(c)$ is surjective. It is injective by Theorem \ref{cprepr2}(d).\vskip.03in

(b)  This is an immediate consequence of (a) and Theorem \ref{cprepr4.5}(b).\vskip.03in

(c)  This follows from (a) and (b).  If $\mathcal W$ is finite, Theorem \ref{cprepr5} says that $\mathcal P(c)=\Phi^+.$ 
\end{proof}

We finish with a characterization of the words in $\mathrm{Red}\,\mathfrak M(c).$

\begin{thm}\label{cprepr9}  
\begin{itemize}
\item[(a)] For any $X\in \mathfrak M(c),$ the word $\rho(X)$ is reduced if and only if  $X=W_\Psi$  for some $\Psi\in\mathfrak i(c).$ Hence $\mathrm{Red}\,\mathfrak M(c)$ is the image of the map $\xi:\mathfrak{i}(c)\to\mathfrak M(c),$ and the map $\xi:\mathfrak i(c)\to\mathrm{Red}\,\mathfrak M(c)$  is bijective.
\item[(b)]  The map $\rho:\mathrm{Red}\,\mathfrak M(c)\to\mathcal W$ is order-embedding, hence, injective.  
\item[(c)]  The maps of (a) and (b) are bijections  between $\mathfrak i(c),$   $\mathrm{Red}\,\mathfrak M(c),$ and the set of elements of $\mathcal W$  having a $c$-admissible reduced expression.
\end{itemize}
\end{thm}
\begin{proof} (a) The  sufficiency is Theorem \ref{cprepr4.5}(a).  By Proposition \ref{independent.2}, $X= X_1\vee\dots\vee X_m$ where $\{X_1,\dots,X_m\}$ is an independent subset of $\mathfrak P(c).$  Suppose  $\rho(X)$ is reduced.  By Lemma \ref{cprepr4}(a), the word $\rho(X_j)$ is reduced for all $j,$ so Theorem \ref{cprepr7}(a) says that $X_j= W_{\alpha_j}$  for some $\alpha_j\in\mathcal P(c).$ By Definition \ref{independent},   $\Psi=\{\alpha_1,\dots,\alpha_m\}$ is an independent subset of $\mathcal P(c),$ so Proposition \ref{independent1} says that $X=W_\Psi$ where $\Psi$ is uniquely determined.

\vskip.03in(b) This is an immediate consequence of (a) and Theorem \ref{cprepr4.5}(d).

\vskip.03in(c) This follows from (a) and (b).
\end{proof}

\begin{rmk}\label{approximation}  For any $\Theta\in\mathfrak f(c),$ Theorem \ref{cprepr2}(a) gives the word $W_\Theta\in\mathfrak M(c).$ By Proposition \ref{independent.2}, there exists an independent subset $\{X_1,\dots,X_m\}$ of $\mathfrak P(c)$ for which $W_\Theta= X_1\vee\dots\vee X_m$  and the $X_i$ are unique up to permutation. Since the word $\rho(W_\Theta)$ is reduced by Theorem \ref{cprepr4.5}(a),  the word $\rho(X_i)$ is reduced for all $i$ by Lemma \ref{cprepr4}(a), and Theorem \ref{cprepr7}(a) says that $X_i= W_{\alpha_i}$ for a unique $\alpha_i\in\mathcal P(c).$  By Definition \ref{independent}, $\Psi=\{\alpha_1,\dots,\alpha_m\}$ is an independent subset of $\mathcal P(c),$ and Proposition \ref{independent1}(a) says that $W_\Theta= W_\Psi,$ while Proposition \ref{independent1}(b) says that $\Psi$ is uniquely determined.  Thus, for any $\Theta\in\mathfrak f(c)$ there exists a unique $\Psi\in\mathfrak i(c)$ satisfying $W_\Theta = W_\Psi.$  One may view $\Psi$ as an \lq\lq approximation" of $\Theta.$
\end{rmk}

\end{document}